\documentclass[abstract=no]{scrartcl} 
\usepackage[margin=3.5cm]{geometry}

\usepackage{amsmath}
\usepackage{bbm}
\usepackage{latexsym}
\usepackage[english]{babel}
\usepackage{amsfonts}
\usepackage[T1]{fontenc}
\usepackage[utf8]{inputenc}
\usepackage{amsthm}
\usepackage{amssymb}
\usepackage{dsfont}
\usepackage{version}
\usepackage{caption}
\usepackage{extarrows}
\usepackage[backend=biber,style=alphabetic,maxbibnames=99]{biblatex}
\usepackage[psamsfonts]{eucal}
\usepackage{stmaryrd}
\usepackage{url}
\usepackage{hyperref}
\usepackage{color}
\usepackage{graphicx}
\usepackage{blindtext}
\usepackage{tikz}
\usetikzlibrary{arrows,matrix,positioning}
\tikzstyle{dmatrix}=[matrix of math nodes,row sep=2.5em, column sep=2.5em,
text height=1.5ex, text depth=0.25ex] 
\usepackage{mathtools}
\usepackage{microtype}
\usepackage{graphicx}
\usepackage[all]{xy}
\usepackage{lipsum}
\usepackage{float}
\usepackage{csquotes}
\usepackage{empheq}
\usepackage{enumitem}

\numberwithin{equation}{section}

\setparsizes{1em}{0.15\baselineskip plus .25\baselineskip}{1em plus 1fil}

\theoremstyle{plain}
\newtheorem{theorem}{Theorem}[section]
\newtheorem{lemma}[theorem]{Lemma}

\newtheorem{cor}[theorem]{Corollary}

\newtheorem{prop/Def}[theorem]{Propsition/Definition}
\newtheorem{theorem/Def}[theorem]{Theorem/Definition}
\theoremstyle{definition}

\newtheorem{Def}[theorem]{Definition}
\newtheorem{rem}[theorem]{Remark}
\newtheorem{exa}[theorem]{Example}

\def \Q {{\mathbb Q}}
\def \R {{\mathbb R}}
\def \N {{\mathbb N}}
\def \Z {{\mathbb Z}}

\def \D {{\pmb{D}}}

\def \X {{ \mathfrak{X}}}
\def \T {{\mathbb T}}
\def \div {{\operatorname{div}}}

\def \vol {{ \operatorname{vol}}}

\def \max {{ \operatorname{max}}}

\def \MV {{ \operatorname{MV}}}

\def \Div {{ \operatorname{Div}}}

\def \sm {{\operatorname{sm}}}

\def \mult {{\operatorname{mult}}}
\def \tor {{\operatorname{tor}}}

\bibliography{bibliography_update}

\title{The convex-set algebra and the toric b-Chow group}
\author{Ana Mar\'ia Botero}
\date{}
\begin{document}

\maketitle

{\small{\begin{abstract}
In \cite{botero}, a top intersection product of toric b-divisors on a smooth complete toric variety is defined. It is shown that a nef toric b-divisor corresponds to a convex set and that its top intersection number equals the volume of this convex set. The goal of this article is to extend this result and define an intersection product of sufficiently positive toric b-classes of arbitrary codimension. For this, we extend the polytope algebra of McMullen (\cite{McM1}, \cite{McM2}, \cite{McM3}) to the so called \emph{convex-set algebra} and we show that it embeds in the toric b-Chow group. In this way, the convex-set algebra can be viewed as a ring for an intersection theory for sufficiently positive toric b-classes. As an application, we show that some Hodge type inequalities are satisfied for the convex-set algebra.
\end{abstract}}}
\tableofcontents
\section{Introduction}
Let $X$ be a smooth, complete variety over an algebraically closed field $k$. Consider the set $R(X)$ of all pairs 
 $\left\{\left(X_{\pi}, \pi\right)\right\}$, where $X_{\pi}$ is a smooth complete variety and $\pi \colon X_{\pi} \to X$ is a proper birational morphism. We endow this set with a partial order by setting $\pi' \geq \pi$ if and only if there exists a proper birational morphism (which is necessarily unique) $\mu \colon X_{\pi'} \to X_{\pi}$ such that $\pi' = \pi \circ \mu$. The \emph{Riemann--Zariski Space} of $X$ is defined as the projective limit 
\[
\X \coloneqq \varprojlim_{\left(X_{\pi}, \pi\right) \in R(X)} X_{\pi},
\]
taken in the category of locally ringed topological spaces, with maps given by the $\mu$'s. See \cite{ZS} and \cite{VA} for a more detailed discussion on the structure of this space, which is introduced here only for illustrative purposes.

Let $A^*(X)_{\Q} = \bigoplus_{\ell} A^{\ell}(X)_{\Q}$ be the Chow ring of $X$, with rational coefficients. 
 Given $\pi' \geq \pi$, we have an induced push-forward map 
\[
\mu_* \colon A^*\left(X_{\pi'}\right)_{\Q} \longrightarrow A^*\left(X_{\pi}\right)_{\Q},
\]
which is just a group homomorphism. 
The \emph{b-Chow group of $\X$} is then defined as the inverse limit 
\[
A^*(\X)_{\Q} \coloneqq \varprojlim_{\left(X_{\pi}, \pi\right) \in R(X)} A^*\left(X_{\pi}\right)_{\Q}
\]
in the category of groups, with maps given by the push-forward maps. An element in the b-Chow group is called a \emph{b-class}. We can think of a b-class as a tuple of classes
\[
\left(c_{\pi}\right)_{(X_{\pi},\pi) \in R(X)}, \quad c_{\pi} \in A^*(X_{\pi})_{\Q},
\]
compatible with the push-forward map, i.e. such that $\mu_*c_{\pi'} = c_{\pi}$ whenever $\pi' \geq \pi$. 

Although in another context and for different purposes, b-Chow groups of Riemann--Zariski spaces appear also in the work of \cite{Al}.

Note that natural constructions of objects in algebraic geometry often give rise to b-classes. For example, the divisor $\div(\varphi)$ of a rational function $\varphi$ on $X$, as well as the divisor $\div(\omega)$ of a rational differential $\omega$ on $X$ make sense as b-divisors. Indeed, their multiplicities $\mult_E\varphi$ and $\mult_E\omega$ along any prime divisor $E$ on any birational model of $X$ are well defined. Also, the Todd class $\operatorname{Td}(X)$ of $X$ is naturally a b-class. This follows from the Grothendieck--Riemann--Roch Theorem. Indeed, given a proper morphsim $f \colon Y \to X$ with Y smooth and complete, and given any coherent sheaf $\mathcal{F}$ on $Y$, the Grothendieck--Riemann--Roch Theorem states that 
\[
f_*\left(ch(\mathcal{F}) \cdot \operatorname{Td}(Y)\right) = ch\left(f_!(\mathcal{F})\right) \cdot \operatorname{Td}(X)
\]
on $A^*(X)_{\Q}$. Here, $ch(\cdot)$ denotes the Chern character and $f_!$ the \emph{generalized Gysin map}, defined as an alternating sum of higher direct images \cite[Theorem 5.2]{Ful2}. Consider $\mathcal{F} = \mathcal{O}_Y$ and $f$ birational. Then, using that $ch(\mathcal{O}_Y) = 1$ and that higher direct images of the structure sheaf vanish, we see that
\[
f_*(\operatorname{Td}(Y))= \operatorname{Td}(X).
\]
Thus, the b-Todd class of $\X$ defined by
\[
\operatorname{Td}(\X) \coloneqq \left(\operatorname{Td}(X_{\pi})\right)_{(X_{\pi}, \pi) \in R(X)}
\]
 is well-defined.

Clearly, the b-version of an algebraic object carries more information than the original object. Hence, it makes sense to try to define an intersection theory on $A^*(\X)_{\Q}$ as this would provide much more refined birational invariants. However, this is in general not possible, as can already be seen in the toric case for top intersection products \cite[Example~4.10]{botero}. However, we may ask if there is a subset of the b-Chow group, defined by a meaningful positivity condition, on which an intersection product is well-defined. The main purpose of this article is to show that in the toric case, this is possible. 

In order to state our main result, we restrict now to the toric case and recall some definitions and results. 

Let $X_{\Sigma}$ be a smooth and complete toric variety over $k$, defined by a smooth and complete fan $\Sigma \subset M_{\R}$. Here, $M$ denotes the character lattice of the algebraic torus acting on $X_{\Sigma}$. Then, the set $R(X_{\Sigma}) = R(\Sigma)$ consisting of all \emph{toric} proper birational morphisms to $X$ corresponds to the set of smooth complete fans refining $\Sigma$. Since toric resolution of singularities is given (see e.g. \cite{ful}), the set $R(\Sigma)$ carries the structure of a directed set. In \cite{botero}, the author developed a top intersection theory of toric $b$-divisors on $X_{\Sigma}$.
 The top intersection number of a toric b-divisor is defined as a limit of top intersection numbers as we vary $\Sigma' \in R(\Sigma)$. It is shown that if the toric b-divisor $\D$ is nef, then it corresponds to a compact convex set $K_{\D}$. Moreover, its top intersection number exists, is finite and equals the volume of $K_{\D}$ (see Section \ref{sec:b-divisors}).

The aim of this article is to extend this top intersection product of nef toric b-divisors to an intersection product of sufficiently positive toric b-classes of arbitrary codimension. To this end, we extend the polytope algebra $\Pi$ of McMullen (\cite{McM1}, \cite{McM2}, \cite{McM3}) to the so called \emph{convex-set algebra} $\mathcal{C}$, generated by classes $[K]$ of some compact convex sets $K$ in $M_{\R}$, modulo some relations (see Definition~\ref{def:convex-set}). This a $\Q$-graded algebra which contains $\Pi$ as a $\Q$-graded subalgebra. 

Now, Fulton and Sturmfels show in \cite[Theorem 4.2]{FS}, that there is an isomorphism of $\Q$-graded algebras
\begin{align}\label{eq:iso-pol}
f\colon \Pi \longrightarrow \varinjlim_{\Sigma' \in R(\Sigma)} A^*\left(X_{\Sigma'}\right)_{\Q},
\end{align}
where the direct limit is taken with respect to the pull-back map. Motivated by this result, we define the toric b-Chow group
\[
A^*\left(\mathfrak{X}_{\Sigma}^{\tor}\right)_{\Q} \coloneqq \varprojlim_{\Sigma' \in R(\Sigma)} A^*\left(X_{\pi}\right)_{\Q},
\]
where the inverse limit is taken with respect to the push-forward map. We show the following theorem which is a combination of Theorem \ref{th:injection} and Corollary \ref{cor:inj} in the text.
\begin{theorem}\label{th:intro}
There is a subgroup $S \subset A^*\left(\mathfrak{X}_{\Sigma}^{\tor}\right)_{\Q}$ endowed with an algebra structure, and an isomorphism of $\Q$-graded algebras 
\begin{align}
\iota \colon \mathcal{C} \longrightarrow S,
\end{align}
satisfying that 
\[
\log\left([K_{\D}]\right) \longmapsto \D,
\]
for a nef toric b-divisor $\D$ on $X_{\Sigma}$ with corresponding convex set $K_{\D}$. Moreover, the restriction to the polytope algebra coincides with the isomorphism from \ref{eq:iso-pol}.
\end{theorem}
It follows that the convex-set algebra can be viewed as a ring for an intersection theory of sufficiently positive toric b-classes on the toric b-Chow group. Moreover, by Lemma~\ref{lem:volume}, in the case of nef toric b-divisors, the top intersection product on $\iota(\mathcal{C})$ induced by the multiplication on $\mathcal{C}$ coincides with the top intersection product defined in \cite{botero}.

As an application, in Corollary \ref{cor:inj2}, we show a structure result for the convex-set algebra $\mathcal{C}$ extending the structure result for the polytope algebra from \cite{Brion}. 

Finally, we show that some Hodge type inequalities are satisfied for $\mathcal{C}$.

In order to contextualize these results, recall from \cite[Sections 5 and 6]{botero} that given any toric b-divisor $\D$ on a toric variety $X$ (in fact you can do this for any b-divisor), one can consider its space of global sections $H^0(X,\D)$. This is a finite-dimensional vector space of rational functions. The graded algebra $\bigoplus_{k \geq 0}H^0(X, k\D)t^k$ is however not finitely generated. But it has the property of being of \emph{almost integral type} and thus, by \cite{KK}, one can associated a convex Okounkov body to it. It turns out that if $\D$ is nef and big, then this Okounkov body is related to the convex body $K_{\D}$ described above \cite[Theorem 6.16]{botero}. Now, if $\D$ is b-Cartier then its space of global sections is finite-dimensional and (at least in the toric case), its corresponding Okounkov body is a polytope. As  was observed in \cite{KK-ca} and \cite{KK2}, in this b-Cartier case, top intersection numbers correspond to mixed multiplicity indexes of finite dimensional spaces of rational functions which in turn correspond to mixed volumes of convex Okounkov bodies. In this way, the intersection theory of sufficiently positive (not necessarily Cartier) toric b-divisors developed in this paper can be seen as a generalization of this top intersection theory of finite-dimensional spaces of rational functions and their corresponding Okounkov bodies.  

After posting these results in the arXiv, we learned that Dang and Favre, independently of our preprint, have come up with an intersection theory of nef b-divisors on projective varieties defined over a countable field. They treat the toric case specifically, however without making use of any convex-set algebra. The relation to this theory will be studied more deeply in subsequent papers. 

The structure of this article is as follows. In Section \ref{sec:conv-set} we give the definition of the convex-set algebra $\mathcal{C}$ and state some of its properties. In Section \ref{sec:b-divisors} we recall the main definitions and integrability properties of toric b-divisors from \cite{botero}. In Section \ref{sec:chow} we define the toric b-Chow group of the toric Riemann--Zariski space and we recall the definition of the isomorphism \eqref{eq:iso-pol}. In Section \ref{sec:ppf} we relate the polytope algebra with spaces of rational piecewise polynomial functions on $N_{\R} = M_{\R}^{\vee}$. Our main theorem is then shown in Section \ref{sec:emb}. Finally, in Section~\ref{sec:applications}, we show that some Hodge type inequalities are satisfied for the convex set algebra $\mathcal{C}$.

\medskip \noindent \textbf {Acknowledgments.} We are grateful to
  the anonymous referee for all her/his constructive remarks.

%

\section{Definitions and notations}\label{sec:not}
Let $k$ be an algebraically closed field of arbitrary characteristic. Let $\mathbb{T} \simeq \mathbb{G}_m^n$ be an $n$-dimensional algebraic torus over $k$. We denote by $M$ the $n$-dimensional character lattice of $\mathbb{T}$ and by $N = M^{\vee}$ its dual lattice. For any ring $R$ we denote by $M_R$, $N_R$ the tensor product $M \otimes_{\Z} R$, resp. $N \otimes_{\Z} R$. 

\begin{Def}
We define the sets $W$ and $W_{\sm}$ by 
\[
W \coloneqq \left\{\Sigma \; \big{|} \; \Sigma \text{ is a complete, rational (generalized) fan in }N_{\R} \right\}
\]
and 
\[
W_{\sm} \coloneqq \left\{\Sigma \; \big{|} \; \Sigma \text{ is a smooth, complete, rational (generalized) fan in }N_{\R} \right\},
\]
respectively. Here, \emph{generalized} means in the sense of \cite[Section 6.2]{CLS}, i.e. the cones are not necessarily required to be strictly convex. A generalized fan that
is an ordinary fan is called non-degenerate; otherwise it is degenerate. We endow $W$ with the partial order given by 
\[
\Sigma'' \geq \Sigma' \quad \text{iff} \quad \Sigma'' \text{ is a refinement of }\Sigma'.
\]

Since for every pair $\Sigma'$, $\Sigma'' \in W$ we can always find a rational refinement $\Sigma''' \geq \Sigma'$, $\Sigma''' \geq \Sigma''$, the set $W$ carries a structure of a directed set.

We set 
\[
W' \coloneqq \left\{ \Sigma \in W \; \big{|} \; \Sigma \text{ is non- degenerate} \right\}\subset W
\]
 with its induced directed set structure. 
 
 Similarly, we endow $W_{\sm}$ with a driected set structure and define
 \[
W'_{\sm} \coloneqq \left\{ \Sigma \in W_{\sm} \; \big{|} \; \Sigma \text{ is non- degenerate} \right\}\subset W_{\sm}
\]
with its induced directed set structure.
\end{Def}
\begin{rem}\label{rem:cofinal}
Since any fan $\Sigma \in W$ has a smooth refinement, the sets $W_{\sm}$ and $W'_{\sm}$ are cofinal in $W$ and $W'$, respectively. 
\end{rem} 

Given a rational polytope $P \subset M_{\R}$, we denote by $\Sigma_P \in W$ its normal fan. 

\begin{rem}\label{rem:proj-fan}
A fan $\Sigma \in W$ is said to be \emph{projective} if it is the normal fan of some polytope. It follows from the toric Chow lemma  \cite[Theorem 6.1.19]{CLS} that the set of projective fans in $W$ form a cofinal subset.
\end{rem}

Throughout this article, we assume familiarity with toric geometry. For a detailed introduction to this rich subject, we refer to \cite{CLS}. 
 
For any fan $\Sigma \in W$ we denote by $X_{\Sigma}$ the corresponding complete toric variety. Note that since we are taking generalized fans, the resulting toric varieties are not necessarily full-dimensional. Also, recall that if $\Sigma$ is smooth, then $X_{\Sigma}$ is a smooth toric variety. Given any toric Cartier divisor $D$ on $X_{\Sigma}$ we denote by $h_D \colon |\Sigma| \to \R$ its corresponding suport function. 

Finally, for a smooth algebraic variety $X$ over $k$ we let $A^*(X)_{\Q} = \bigoplus_{\ell} A^{\ell}(X)_{\Q}$ be its Chow ring with rational coefficients. This carries a natural graded algebra structure.


\section{The convex-set algebra}\label{sec:conv-set}
Let notations be as in Section \ref{sec:not}. We give the definition of the convex-set algebra $\mathcal{C}$ and state some of its properties. 

Recall that a non-empty subset $K \subset M_{\R}$ is \emph{convex} if for each pair of points $m_1, m_1 \in K$, the line segment
\[
[m_1, m_2] = \left\{tm_1 + (1-t)m_2 \, \big{|} \, 0 \leq t \leq 1 \right\}
\]
is contained in $K$. Examples of convex sets are cones and polyhedra. 

We now recall some definitions and properties of convex sets. For more a more detailed introduction to convex geometry we refer to \cite{ROCK}.
\begin{Def}
 The \emph{support function} of a convex set $K \subset M_{\R}$ is the function 
\[
h_K \colon N_{\R} \longrightarrow \underline{\R} \; \left(\coloneqq \R \cup \{-\infty\}\right)
\]
given by the assignment
\[
 v \longmapsto\inf_{m \in K}\langle m,v\rangle.
\] 
\end{Def}
Support functions of convex sets are concave, upper semi-continuous and conical, i.e. they satisfy $h_K(\lambda v) = \lambda h_K(v)$ for any non-negative real number $\lambda$.

Conversely, given a concave, upper semi-continuous and conical function $f \colon N_{\R} \to \underline{\R}$, one defines the \emph{stability set} $K_f \subset M_{\R}$ by 
\[
K_f \coloneqq \left\{x \in M_{\R} \; \big{|} \; \langle x,u \rangle - f(u) \; \text{ is bounded below } \forall u \in N_{\R} \right\}.
\]
This is a compact convex set. 

We have 
\[
h_{K_f} = f \quad \text{ and } \quad K_{h_K} = \overline{K},
\]
hence these operations give a bijection between compact convex sets in $M_{\R}$ and concave, upper semi-continuous and conical functions on $N_{\R}$. 
\begin{Def}\label{def:rational_convex}
A convex set $K \subset M_{\R}$ is said to be \emph{rational} if its support function $h_K$ satisfies $h_K\left(N_{\Q}\right) \subset \Q$.

\end{Def}
Note that if $K = P$ is a polytope, then being rational means that its vertices have rational coordinates.
\begin{exa}
Consider the function $h \colon \R^2 \to \R$ defined by 
\[
h(x,y) = \begin{cases} \frac{xy}{x+y}, \; \text{ if }x,y \geq 0 \text{ and }x+y > 0, \\
\min(x,y), \; \text{ otherwise.} \end{cases}
\]
Then $h$ is a concave, conical function. Its corresponding rational compact convex set is 
\[
K_h = \left\{(x,y) \in \left(\R^2\right)^{\vee} \; \big{|} \; x,y \geq 0 \; , \; x+y \leq 1 \; , \; \sqrt{x} + \sqrt{y} \geq 1 \right\}.
\]
\end{exa}

Recall that the normal fan of a rational polytope $P \subset M_{\R}$ is a fan in $W$ denoted by $\Sigma_P$. We define the following two sets.
\begin{Def}
Let \[
\mathcal{K} \coloneqq \left\{ K \subset M_{\R} \; \big{|} \; K \text{ rational compact convex set}\right\} \cup \{\emptyset\}
\] \
and 
\[
\mathcal{P} \coloneqq \left\{ P  \subset M_{\R} \; \big{|} \; P \text{ rational polytope} \right\} \cup \{\emptyset\}.
\]
Clearly, we have the inclusion $\mathcal{P} \subset \mathcal{K}$. 
We say that 
\[
P'' \geq P' \text{ in } \mathcal{P}\quad \text{iff} \quad \Sigma_{P''} \text{ is a refinement of }\Sigma_{P'} \text{ in }W.
\]
 Modulo rational translations, this is a partial order and gives $\mathcal{P}$ the structure of a directed set. 
 \end{Def}
 The following is the definition of McMullen's polytope algebra \cite{Brion}, \cite{FS}, \cite{McM2}.
\begin{Def}
 The polytope algebra $\Pi$ is the $\Q$-algebra with a generator $[P]$ for every polytope $P \in \mathcal{P}$ and $[\emptyset] = 0$, subject to the following relations:
 \begin{itemize}
 \item $[P_1 \cup P_2] + [P_1 \cap P_2] = [P_1] + [P_2]$ whenever $P_1 \cup P_2 \in \mathcal{P}$.
 \item $[P + t] = [P]$ for all translations $t \in \Q^n$.
 \item $[P_1] \cdot [P_2] = [P_1 + P_2]$,
 where $P_1 + P_2$ denotes the Minkowski sum of convex sets. 
 \end{itemize}
\end{Def}

Before defining the convex-set algebra $\mathcal{C}$, we make $\mathcal{K}$ into a metric space by endowing it with the Hausdorff metric \cite[Section 1.8]{BM}. This is defined by 
\[
d_H(K,L) \coloneqq \max\left\{\sup_{x \in K}\inf_{y \in L} |x-y|, \sup_{x \in L}\inf_{y \in K}|x-y|\right\},
\]
for $K,L$ in $\mathcal{K}$. Here, $|\cdot |$ denotes the Euclidean metric on $M_{\R}$ (after choosing an isomorphism $M_{\R} \simeq \R^n$). 

Note that $\mathcal{K}$ is closed under rational translations, finite intersections, finite unions (if convex) and Minkowski sums. 

 We are now ready to extend the polytope algebra $\Pi$ to the convex-set algebra $\mathcal{C}$. 
 \begin{Def}\label{def:convex-set}
 The convex-set algebra $\mathcal{C}$ is the $\Q$-algebra with a generator $[K]$ for every compact convex set $K \in \mathcal{K}$ and $[\emptyset] = 0$, subject to the following relations:
 \begin{itemize}
 \item $[K_1 \cup K_2] + [K_1 \cap K_2] = [K_1] + [K_2]$ whenever $K_1 \cup K_2 \in \mathcal{K}$.
 \item $[K + t] = [K]$ for all translations $t \in \Q^n$.
 \item $[K_1] \cdot [K_2] = [K_1 + K_2]$,
 where $K_1 + K_2$ denotes the Minkowski sum of convex sets. 
 \end{itemize}
\end{Def}

We clearly have 
\[
\Pi \subset \mathcal{C}
\]
as $\Q$-algebras. We now list some properties of $\mathcal{C}$ which extend known properties of $\Pi$.
\begin{enumerate}
\item The multiplicative unit is the class of a point $1 = \left[\{0\}\right]$.

\item For any $K \in \mathcal{K}$, we have 
\begin{align}\label{eq:nilpotent}
\left([K]-1\right)^{n+1} = 0.
\end{align}
Indeed, in the polytopal case, this is \cite[Lemma 13]{McM1}. To see the non-polytopal case, consider the Hausdorff metric $d_H$ on $\mathcal{K}$ as defined above. By \cite[Theorem 2.4.15]{BM}, given a compact convex set $K \in \mathcal{K}$, there exists a sequence of polytopes $(P_i)_{i \in \N}$ in $\mathcal{P}$ converging to $K$ with respect to $d_H$. Moreover, by \cite[Section 3, pg. 139]{BM}, the Minkowski addition of elements in $\mathcal{K}$ is continuous with respect to $d_H$. Also, rational translations are continuous with respect to this metric. Hence, $d_H$ induces a metric on $\mathcal{C}$ and we have $[K]= \lim_{i \in \N}[P_i]$.  Then
\[
\left([K] -1\right)^{n-1} = \left(\lim_{i \in \N}\left[P_i\right]-1\right)^{n-1}=\lim_{i \in \N}\left(\left[P_i\right]-1\right)^{n-1} = 0.
\]

It follows that the logarithm of the class $[K]$ given by
\begin{align}\label{eq:log}
\log\left([K]\right) \coloneqq \sum_{r=1}^n(-1)^{r+1}\frac{1}{r}\left([K]-1\right)^r
\end{align}
is well defined.
\item $\mathcal{C}$ has a structure of a graded $\Q$-algebra 
\[
\mathcal{C} = \bigoplus_{\ell=0}^{\infty} \mathcal{C}_{\ell},
\]
where the $\ell$'th graded component $\mathcal{C}_{\ell}$ is the $\Q$-vector space spanned by all elements of the form $\left(\log[K]\right)^{\ell}$, for $K$ running through all compact convex sets in $\mathcal{K}$. 


\end{enumerate}

\begin{rem} The grading on $\mathcal{C}$ is the direct generalization of the grading on the polytope algebrta $\Pi$ \cite{Brion}. It is explained by Corollaries \ref{cor:inj} and \ref{cor:inj2}. Recall that $n$ denotes the dimension of the vector space $M_{\R}$. In the case of the polytope algebra, it  is shown in \cite{Brion} that the graded components $\Pi_k$ vanish for $k > n$. We have a similar result for $\mathcal{C}$ in Corollary \ref{cor:inj2}. 
\end{rem}
The following definition is taken from \cite[Section 4]{FS}.
\begin{Def}\label{def:subalg}
Let $P \in \mathcal{P}$ be a polytope. Then $\Pi(P)$ is the $\Q$-subalgebra of $\Pi$ generated by all classes $[Q] \in \Pi$, such that $Q \in \mathcal{P}$ is a Minkowski summand of $P$, i.e. such that $P = \lambda Q + R$ for some $\lambda \in \Q_{>0}$ and some polytope $R \in \mathcal{P}$. 
\end{Def}
We make the following remarks concerning the subalgebra $\Pi(P)$ \cite[Proposition 6.2.13]{CLS}. 
\begin{rem}\label{rem:toric-nef}
\begin{enumerate}
\item Let $Q \in \mathcal{P}$. The class $[Q]$ belongs to $\Pi(P)$ if and only if the normal fan $\Sigma_P$ of $P$ is a refinement of the normal fan $\Sigma_Q$ of $Q$. Hence, we can say that $\Pi(P)$ is the $\Q$-subalgebra generated by the classes 
\begin{align}\label{eq:gen}
\left\{ [Q] \; \big{|} \; Q \in \mathcal{P} \; \text{and} \; \Sigma_P \geq \Sigma_Q \text{  in }W \right\}.
\end{align}
Alternatively, by the classes
\[
\left\{ [Q] \; \big{|} \; Q \in \mathcal{P} \; \text{and}\; P \geq Q \text{ in }\mathcal{P} \right\}.
\]
\item Recall that to any toric Weil $\Q$-divisor $D = \sum_{\tau \in \Sigma_P(1)}a_{\tau}D_{\tau}$ on the projective toric variety $X_{\Sigma_P}$, one can attach the rational polytope
\[
P_D \coloneqq \left\{m \in M_{\R} \, | \, \langle m, v_{\tau} \rangle \geq -a_{\tau} \right\} \subset M_{\R},
\]
where $v_{\tau}$ denotes the primitive vector spanning the ray $\tau$. 
Then, given $Q \in \mathcal{P}$, we have that $[Q] \in \Pi(P)$ if and only if there exists a nef toric divisor $D$ on $X_{\Sigma_P}$ such that $Q$ is the polytope associated to $D$, i.e. such that $Q = P_D$. Hence, there is a bijection between the distinguished set of generators \eqref{eq:gen} and nef toric $\Q$-divisors on $X_{\Sigma_p}$. 
 \end{enumerate} 
\end{rem}
\section{Intersection theory of toric b-divisors on toric varieties}\label{sec:b-divisors}
Let notations be as in Section \ref{sec:not}. We recall the main definitions and integrability properties of toric b-divisors. For a more detailed introduction to this subject we refer to \cite{botero} (see also \cite{botero2}).

Let $\Sigma \in W_{\sm}'$ be a smooth, non-degenerate complete fan in $N_{\R}$ and let $X_{\Sigma}$ be its corresponding smooth and complete $n$-dimensional toric variety. Let $W_{\sm}'(\Sigma) \subset W_{\sm}'$ consist of all smooth subdivisions of $\Sigma$ with its induced directed set structure. The \emph{toric Riemann--Zariski Space} of $X_{\Sigma}$ is defined as the inverse limit 
\[
\mathfrak{X}^{\tor}_{\Sigma} \coloneqq \varprojlim_{\Sigma' \in W_{\sm}'(\Sigma)}X_{\Sigma'},
\] 
with maps given by the toric proper birational morphisms $\pi\colon X_{\Sigma''} \to X_{\Sigma'}$ induced whenever $\Sigma'' \geq \Sigma'$ in $W_{\sm}'(\Sigma)$. 

Given $\Sigma' \in W'_{\sm}(\Sigma)$, we denote by $ \mathbb{T}\text{-}\Div(X_{\Sigma'})_{\Q}$ the set of toric $\Q$-divisors on $X_{\Sigma'}$. The group of \emph{toric Cartier b-divisors} on $\X^{\tor}_{\Sigma}$ is defined as the direct limit
\[
\text{Ca}\left(\mathfrak{X}^{\tor}_{\Sigma}\right)_{\Q} \coloneqq \varinjlim_{\Sigma' \in W_{\sm}'(\Sigma)} \mathbb{T}\text{-}\Div(X_{\Sigma'})_{\Q},
\]
with maps given by the pull-back maps of toric divisors.

The group of \emph{toric Weil b-divisors} on $\X^{\tor}_{\Sigma}$ is defined as the inverse limit
\[
\text{We}\left(\mathfrak{X}^{\tor}_{\Sigma}\right)_{\Q} \coloneqq \varprojlim_{\Sigma' \in W_{\sm}'(\Sigma)} \mathbb{T}\text{-}\Div(X_{\Sigma'})_{\Q},
\]
with maps given by the push-forward maps of toric divisors. 

We will denote b-divisors with a bold $\D$ to distinguish them from classical divisors $D$.

We have 
\[
\text{Ca}\left(\mathfrak{X}^{\tor}_{\Sigma}\right)_{\Q} \subset \text{We}\left(\mathfrak{X}^{\tor}_{\Sigma}\right)_{\Q}.
\]

More precisely, we can think of a Weil toric b-divisor as a net of toric $\Q$-divisors 
\[
\D = \left(D_{\Sigma'}\right)_{\Sigma' \in W_{\sm}'(\Sigma)},
\]
satisfying that $\pi_*D_{\Sigma''} = D_{\Sigma'}$ whenever $\Sigma'' \geq \Sigma'$. Then a Cartier toric b-divisor is a Weil toric b-divisor $\D$ as above which is determined on some $\widetilde{\Sigma} \in W'_{\sm}(\Sigma)$, i.e. such that for any other $\Sigma' \geq \widetilde{\Sigma}$ in $W'_{\sm}(\Sigma)$, we have that $D_{\Sigma'} = \pi^*D_{\widetilde{\Sigma}}$.

We now define the positivity notion which allows us to define top intersection numbers of toric b-divisors. 
\begin{Def}
A toric b-divisor $\D = \left(D_{\Sigma'}\right)_{\Sigma' \in W_{\sm}'(\Sigma)}$ is \emph{nef}, if $D_{\Sigma'} \in \mathbb{T}\text{-}\Div(X_{\Sigma'})$ is nef for all $\Sigma'$ in a cofinal subset of $W_{\sm}'(\Sigma)$. 
\end{Def}
It follows from basic toric geometry that there is a bijective correspondence between the set of nef toric b-divisors on $X_{\Sigma}$ and the set of $\mathbb{Q}$-valued, conical, $\Q$-concave functions on $N_{\Q}$ (see \cite[Remark 3.7]{botero}). 

\begin{Def}
The \emph{mixed degree} $\D_1 \dotsm \D_n$ of a collection of toric b-divisors is defined as the limit (in the sense of nets)
\[
      \D_1\dotsm \D_n \coloneqq \lim_{\Sigma' \in W_{\sm}'(\Sigma)} D_{1_{\Sigma'}} \dotsm D_{n_{\Sigma'}}
      \]
of top intersection numbers of toric divisors, provided this limit exists and is finite. In particular, if $\D = \D_1 = \dotsc = \D_n$, then the limit (in the sense of nets) 
\[
\D^n \coloneqq \lim_{\Sigma' \in W_{\sm}'(\Sigma)} D_{\Sigma'}^n, 
\]
provided this limit exists and is finite, is called the \emph{degree} of the toric b-divisor $\D$. A toric b-divisor whose degree exists, is said to be \emph{integrable}. 
\end{Def}

Now, the \emph{mixed volume} of a collection of convex sets $K_1, \dotsc, K_n$ in $M_{\R}$ is defined by 
\[
\MV\left(K_1, \dotsc,K_n\right) \coloneqq \sum_{j=1}^n(-1)^{n-j}\sum_{1\leq i_1<\dotsb <i_j\leq n}\vol\left(K_{i_1}+ \dotsb + K_{i_j}\right), 
\]
where the \enquote{$+$} refers to the Minkowski addition of convex sets.

Recall the definition of the stability set of a concave function from Section \ref{sec:conv-set}. The following theorem relates the mixed degree $\D_1 \dotsc \D_n$ of a collection of nef toric b-divisors with the mixed volume of convex sets. It is a combination of \cite[Theorems 4.9 and 4.12]{botero}.
    
\begin{theorem}\label{the:volume}

      Let $\D_1, \dotsc,\D_n$ be a collection of nef toric b-divisors on $X_{\Sigma}$ and let $\tilde{\phi}_i \colon N_{\mathbb{Q}} \to \mathbb{Q}$ be the corresponding $\Q$-concave functions for $i = 1, \dotsc, n$. Then the functions $\tilde{\phi}_i$ extend to conical concave functions ${\phi}_i \colon N_{\mathbb{R}} \to \mathbb{R}$.
      The mixed degree $\D_1\dotsm \D_n$ exists, and is given by the mixed volume of the stability sets $K_{{\phi}_i}$ of the concave conical functions ${\phi}_i$, i.e. we have that
      \[
      \D_1\dotsm \D_n = \MV\left(K_{{\phi}_1}, \dotsc, K_{{\phi}_n}\right).
      \]
      In particular, a nef toric b-divisor $\D$ is integrable, and its degree is given by 
      \[
      \D^n = n!\,\vol\left(K_{{\phi}}\right),
      \]
 where ${\phi}$ denotes the corresponding concave conical function. 
 
 \end{theorem}
 
 \begin{rem}\label{rem:convex-b} The previous theorem associates a compact convex set $K_{\D} \in \mathcal{K}$ to a nef toric b-divisor $\D$. As its support function takes rational values on $N_{\Q}$, we have that $K_{\D}$ is a rational compact convex set.  Conversely, from \cite[Proposition~5.1]{botero}, we have that every rational compact convex set (modulo translations) arises in this way. 
We conclude that there is a set theoretical bijection
\[
K \longmapsto \D_K
\]
 between rational compact convex sets and nef toric b-divisors on all toric varieties $X_{\Sigma}$ for $\Sigma \in W'_{\sm}$. 
\end{rem}
 
\section{The toric b-Chow group}\label{sec:chow}
Let notations be as in Section \ref{sec:not}. We define the toric b-Chow group and the Cartier b-Chow ring of the toric Riemann--Zariski space and recall the relation between the latter and the polytope algebra. 

Let $\Sigma'' \geq \Sigma'$ in $W_{\sm}'$ and let $\pi \colon X_{\Sigma''} \to X_{\Sigma'}$ be the induced toric proper birational morphism. As in the case of divisors, we obtain both a push-forward map
\begin{align}\label{eq:push-cycles}
\pi_* \colon A^*\left( X_{\Sigma''} \right)_{\Q} \longrightarrow A^*\left( X_{\Sigma'} \right)_{\Q}
\end{align}
and a pull-back map
\begin{align}\label{eq:ppull-cycles}
\pi^* \colon A^*\left( X_{\Sigma'} \right)_{\Q} \longrightarrow A^*\left( X_{\Sigma''} \right)_{\Q}
\end{align}
between the Chow rings. The push-forward map is just a group homomorphism whereas the pull-back map preserves the ring structures.

\begin{Def}\label{def:tor-b-chow}
The \emph{toric b-Chow group} $A^*\left(\mathfrak{X}^{\tor}\right)_{\Q}$ of the toric Riemann--Zariski space $\X^{\tor}$
is defined as the projective limit 
\[
A^*\left(\mathfrak{X}^{\tor}\right)_{\Q}\coloneqq \varprojlim_{\Sigma \in W_{\sm}'} A^*\left(X_{\Sigma}\right)_{\Q}
\]
in the category of groups, with maps given by the push-forward maps \eqref{eq:push-cycles}. Elements in $A^*\left(\mathfrak{X}^{\tor}\right)_{\Q}$ are called \emph{toric b-classes}. 

The \emph{toric Cartier b-Chow ring} $A_{\operatorname{Ca}}^*\left(\mathfrak{X}^{\tor}\right)_{\Q}$ of the toric Riemann--Zariski space $\X^{\tor}$
is defined as the injective limit 
\[
A_{\operatorname{Ca}}^*\left(\mathfrak{X}^{\tor}\right)_{\Q}\coloneqq \varinjlim_{\Sigma \in W_{\sm}'} A^*\left(X_{\Sigma}\right)_{\Q}
\]
in the category of rings, with maps given by the pull-back maps \eqref{eq:ppull-cycles}. Elements in $A_{\operatorname{Ca}}^*\left(\mathfrak{X}^{\tor}\right)_{\Q}$ are called \emph{toric Cartier b-classes}. 

The toric Cartier b-Chow ring has a graded algebra structure 
\[
A_{\operatorname{Ca}}^*\left(\mathfrak{X}^{\tor}\right)_{\Q} = \bigoplus_{\ell = 0}^n A_{\operatorname{Ca}}^{\ell}\left(\mathfrak{X}^{\tor}\right)_{\Q},
\]
where 
\[
A_{\operatorname{Ca}}^{\ell}\left(\mathfrak{X}^{\tor}\right)_{\Q} = \varinjlim_{\Sigma \in W_{\sm}'} A^{\ell}\left(X_{\Sigma}\right)_{\Q}.
\]
 
On the other hand, the b-Chow group $A^*\left(\mathfrak{X}^{\tor}\right)_{\Q}$ decomposes as a direct sum of (possibly infinte dimensional) vector spaces
\[
A^*\left(\mathfrak{X}^{\tor}\right)_{\Q} = \bigoplus_{\ell = 0}^nA^{\ell}\left(\mathfrak{X}^{\tor}\right)_{\Q},
\]
where 
\[
A^{\ell}\left(\mathfrak{X}^{\tor}\right)_{\Q} = \varprojlim_{\Sigma \in W'_{\sm}}A^{\ell}\left(X_{\Sigma}\right).
\]

To see that both $A_{\operatorname{Ca}}^{\ell}\left(\mathfrak{X}^{\tor}\right)_{\Q}$ and $A^{\ell}\left(\mathfrak{X}^{\tor}\right)_{\Q}$ vanish in degrees $\ell \geq n$, note that we may restrict to projective fans and hence, to projective toric varieties. Indeed, these form a cofinal subset in $W'$ (see Remark \ref{rem:proj-fan}).

\end{Def}
\begin{rem} 
There is a natural embedding $A_{\operatorname{Ca}}^*\left(\mathfrak{X}^{\tor}\right)_{\Q} \subset A^*\left(\mathfrak{X}^{\tor}\right)_{\Q}$. Moreover, it follows from the projection formula that $A_{\operatorname{Ca}}^*\left(\mathfrak{X}^{\tor}\right)_{\Q}$ acts on $A^*\left(\mathfrak{X}^{\tor}\right)_{\Q}$ making the latter a module over the former.
\end{rem}

We now state known important relations between the toric Cartier b-Chow ring and the polytope algebra $\Pi$. The following theorem follows from \cite[Theorem 4.1]{FS} (see also \cite[Theorem 14.1]{McM2}).
\begin{theorem}\label{th:iso}
Let $P \in \mathcal{P}_{\sm}$. There exists an isomorphism of graded $\Q$-algebras 
\[
\Theta_P \colon \Pi(P) \longrightarrow A^*\left(X_{\Sigma_P}\right)_{\Q}
\]
satisfying that 
\[
[Q] \longmapsto \exp(D_Q) \coloneqq \sum_{r=0}^{\dim X_{\Sigma_P}} \frac{D_Q^r}{r!}
\]
for the distinguished set of generators from \eqref{eq:gen}. Here, $D_Q$ is the nef toric divisor on $X_{\Sigma_P}$ corresponding to $Q$.
\end{theorem}
Note that for any $P_2 \geq P_1$ in $\mathcal{P}_{\sm}$ we have natural inclusions $\Pi(P_1) \hookrightarrow \Pi(P_2)$. Hence, it is natural to consider the direct limit $\varinjlim_{P \in \mathcal{P}} \Pi(P)$ with respect to these inclusions. This is clearly equal to $\Pi$. As it turns out, these inclusion morphsims are compatible with the pull-back morphisms between the Chow groups, i.e. if $P_1$ and $P_2$ are of dimension $n$, then the following diagram commutes.
\begin{figure}[h]
\begin{center}
    \begin{tikzpicture}
      \matrix[dmatrix] (m)
      {
        \Pi(P_1) & A^*\left(X_{\Sigma_{P_1}}\right)_{\Q}\\
       \Pi(P_2) & A^*\left(X_{\Sigma_{P_2}}\right)_{\Q} \\
      };
      \draw[->] (m-1-1) to node[above]{$\Theta_{P_1}$} (m-1-2);
      \draw[->] (m-2-1) to node[below]{$\Theta_{P_2}$} (m-2-2);
      \draw[right hook ->] (m-1-1)--(m-2-1);
      \draw[->] (m-1-2) to node[right]{$\pi^*$} (m-2-2);
     \end{tikzpicture}
     \end{center}
     \end{figure}
     
 Thus, the isomorphisms $\Theta_P$, as we range over all polytopes $P \in \mathcal{P}_{\sm}$, induce an isomorphsim 
\begin{eqnarray}\label{eq:poly-iso}
f \colon \Pi \longrightarrow A_{\operatorname{Ca}}^*\left(\mathfrak{X}^{\tor}\right)_{\Q}
\end{eqnarray}
\cite[Theorem 4.2]{FS}.
 
It follows that the polytope algebra $\Pi$ can be viewed as a ring of intersection theory for toric classes on all (smooth) toric compactifications of the torus $\T$.

The main goal of this article is to define a ring of intersection theory for sufficiently positive toric b-classes which are not necessarily Cartier. In order to do this, as a main tool, we study spaces of rational piecewise polynomial functions in the next section. 

\section{Relation with spaces of rational piecewise polynomial functions}\label{sec:ppf}
Let notations be as in Section \ref{sec:not}. For a given polytope $P \in \mathcal{P}_{\sm}$, we relate the spaces $\Pi(P)$ and $A^*\left(X_{\Sigma_P}\right)_{\Q}$ with spaces of rational piecewise polynomial functions on $N_{\R}$. We mainly follow \cite{Pay} (see also \cite[Section 2]{Brion}).

\subsection*{The ring of rational piecewise polynomial functions}   
   Let $\Sigma \in W'_{\sm}$. For a cone $\sigma \in \Sigma$ we set $M_{\sigma} = M/\left(\sigma^{\perp} \cap M\right)$ and $M_{\sigma, \Q} =  M_{\sigma} \otimes_{\Z} \Q $. The ring of \emph{rational piecewise polynomial functions} $R_{\Sigma}$ is given by
   \[
   R_{\Sigma} \coloneqq \left\{f \colon |\Sigma| \to \R \text{ continuous } \big{|} \; f|_{\sigma} \in \operatorname{Sym} \left(M_{\sigma, \Q} \right) \text{ for each } \sigma \in \Sigma \right\}.
   \]
  The map $f \mapsto \left(f|_{\sigma}\right)_{\sigma \in \Sigma}$ identifies $R_{\Sigma}$ with a subring of $\bigoplus_{\sigma \in \Sigma} \operatorname{Sym} \left(M_{\sigma, \Q} \right)$,
  \[
  R_{\Sigma}\simeq \left\{ \left(f_{\sigma}\right)_{\sigma \in \Sigma} \; \big{|} \; f_{\tau} = f_{\sigma}|_{\tau} \text{ for } \tau \prec \sigma \right\}.
  \]
Here, the notation $\tau \prec \sigma$ denotes that $\tau$ is a face of $\sigma$.
  \begin{rem} In \cite{Pay} the ring of \emph{integral} piecewise polynomial functions is defined, wheras in \cite{Brion} no rationality condition is given. We adapt these definitions to the rational case since we are dealing with rational polytopes and $\Q$-coeffients. 
  \end{rem}

The set $R_{\Sigma}$ has a structure of a graded $\Q$-algebra over $M_{\Q}$ (where the grading is given by the degree). It is called the \emph{algebra of rational piecewise polynomial functions on} $\Sigma$. 

Now, for any ray $\tau \in \Sigma(1)$, the semigroup $\tau \cap \N$ has a unique integral generator $v_{\tau}$. Since $\Sigma$ is smooth, we have that any continuous and piecewise linear function on the support of $\Sigma$ is uniquely defined by its values at the $v_{\tau}$, $\tau \in \Sigma(1)$. In particular, there exists a unique continuous, piecewise linear function $\phi_{\tau}$ such that $\phi_{\tau}\left(v_{\tau}\right) = 1$ and $\phi_{\tau}\left(v_{\tau'}\right) = 0$ for all $\tau' \neq \tau$. 
\begin{Def} Let $\Sigma \in W'_{\sm}$. For any cone $\sigma \in \Sigma$ we set 
\[
\phi_{\sigma} \coloneqq \Pi_{\tau \in \sigma(1)}\phi_{\tau}.
\]
\end{Def}
Then $\phi_{\sigma}$ is a rational piecewise polynomial function, homogeneous of degree $\dim(\sigma)$ and which vanishes outside the star of $\sigma$ in $\Sigma$.

The following definition is adapted from \cite[Section 2.3]{Brion}.
\begin{Def}\label{def:push-functions}
Let $\Sigma' \geq \Sigma$ in $W'_{\sm}$. We define the map
\[ 
\pi_{\Sigma', \Sigma} \colon R_{\Sigma'} \longrightarrow R_{\Sigma}
\]
by
\[
f = \left( f_{\sigma'} \right)_{\sigma' \in \Sigma'} \longmapsto \left( \pi_{\Sigma', \Sigma}(f)_{\sigma} \right)_{\sigma \in \Sigma}
\]
where
\[
 \pi_{\Sigma', \Sigma}(f)_{\sigma} = \phi_{\sigma} \sum_{\substack{\sigma' \subset \sigma \\ \dim(\sigma') = \dim(\sigma)}} \frac{f_{\sigma'}}{\phi_{\sigma'}}.
\]
\end{Def}
The next theorem follows from \cite[Theorem 2.3 and Corollary 2.3]{Brion}.
\begin{theorem}\label{th:comb-push}
Let $\Sigma' \geq \Sigma$ in $W'_{\sm}$. The map $\pi_{\Sigma', \Sigma}\colon R_{\Sigma'} \to R_{\Sigma}$ is well defined and satisfies the following properties:
\begin{enumerate}
\item $\pi_{\Sigma', \Sigma}(1) = 1$.
\item $\pi_{\Sigma', \Sigma}$ is $R_{\Sigma}$-linear.
\item $\pi_{\Sigma', \Sigma}$ is homogeneous of degree $0$.
\item $\pi_{\Sigma',\Sigma} \circ \pi_{\Sigma'', \Sigma'} = \pi_{\Sigma'', \Sigma}$
 for any $\Sigma'' \geq \Sigma' \geq \Sigma$ in $W'_{\sm}$.
\end{enumerate}
Moreover, $\pi_{\Sigma', \Sigma}$ is uniquely determined by the properties $(1)$--$(3)$. 
\end{theorem}


\begin{Def}
Let $\Sigma \in W'_{\sm}$. Consider $M_{\Q}R_{\Sigma} \subset R_{\Sigma}$, the graded ideal generated by all (global) rational linear functions on $N_{\R}$. We define the quotient 
\[
\overline{R}_{\Sigma} \coloneqq R_{\Sigma}/\left(M_{\Q}R_{\Sigma}\right),
\]
which inherits the structure of a graded $\Q$-algebra.
\end{Def}

Now, for $\Sigma' \geq \Sigma$ in $W'_{\sm}$, the map $\pi_{\Sigma',\Sigma} \colon R_{\Sigma'} \to R_{\Sigma}$ from Theorem \ref{th:comb-push} induces a map 
\[
\overline{R}_{\Sigma'} \longrightarrow \overline{R}_{\Sigma},
\]
which we also denote by $\pi_{\Sigma',\Sigma}$. It satisfies the same properties $(1)$--$ (4)$ from above but with $R_{\Sigma}$-linear replaced by $\overline{R}_{\Sigma}$-linear.

\subsection*{Relations to Chow rings and to the polytope algebra}

Let $\Sigma \in W'_{\sm}$. The next theorem relates the Chow ring of the toric variety $X_{\Sigma}$ to the graded $\Q$-algebra $\overline{R}_{\Sigma}$. Recall that $n$ is the dimension of $X_{\Sigma}$.
\begin{theorem}\label{th:chow}
Let $\Sigma \in W'_{\sm}$. There is an isomorphism of graded $\Q$-algebras
\[
\alpha_{\Sigma}\colon A^*\left(X_{\Sigma}\right)_{\Q} \longrightarrow \overline{R}_{\Sigma},
\]
which takes the class of a toric divisor $D$ of $X_{\Sigma}$ to the class of its corresponding support function $h_D$. 
\end{theorem}
\begin{proof}
Consider the \emph{equivariant Chow ring} $A^*_{\operatorname{T}}(X_{\Sigma})$ of the smooth toric variety $X_{\Sigma}$ as defined in \cite{Pay}. 

For $\sigma \in \Sigma$ let $O_{\sigma} \subset X_{\Sigma}$ be the associated toric orbit. As is explained in \emph{loc.~cit.}, there is a natural isomorphism $A^*_{\operatorname{T}} \left(O_{\sigma}\right)_{\Q} \simeq \operatorname{Sym}\left(M_{\sigma,\Q}\right)$. Indeed, for $u \in M_{\Q}$, its image in $M_{\sigma,\Q}$ is identified with the first equivariant Chern class of the equivariant line bundle $\mathcal{O}_{X_{\Sigma}}\left(\operatorname{div}\chi^u\right)|_{O_{\sigma}}$ in $A_{\operatorname{T}}^1\left(O_{\sigma}\right)$.

 Let $\iota_{\sigma} \colon O_{\sigma} \hookrightarrow X_{\Sigma}$ be the inclusion morphism. By \cite[Theorem 1]{Pay}, the map
\[
\bigoplus_{\sigma \in \Sigma}\iota_{\sigma}^* \colon A^*_{\operatorname{T}}\left(X_{\Sigma}\right)_{\Q} \longrightarrow R_{\Sigma}
\]
is an isomorphsim. Moreover, by \cite[Section 2.3]{Brion}, there is a natural isomorphism 
\[
A^*_{\operatorname{T}}\left(X_{\Sigma}\right)_{\Q} / M_{\Q} A^*_{\operatorname{T}}\left(X_{\Sigma}\right)_{\Q} \simeq A^* \left(X_{\Sigma}\right)_{\Q}.
\]
Thus, the theorem follows. 
\end{proof}
We obtain the following corollary (see also \cite[Section 2]{Brion}). 
\begin{cor} 
The graded algebra $\overline{R}_{\Sigma} = \oplus_{\ell = 0}^{\infty}\overline{R}_{\Sigma, \ell}$ vanishes on all degrees $\ell > n$. Moreover, the $\Q$-vector space $\overline{R}_{\Sigma,n}$ is one-dimensional and multiplication in $\overline{R}_{\Sigma}$ induces non-degenrated pairings $\overline{R}_{\Sigma, j} \times \overline{R}_{\Sigma, n-j} \to \overline{R}_{\Sigma, n}$ for $1 \leq j \leq n-1$.
\end{cor}

For any function class $\varphi \in \overline{R}_{\Sigma}$ of degree one, we set 
\[
\exp(\varphi) = \sum_{i = 0}^{n} \varphi^i/i! \in \overline{R}_{\Sigma}.
\]
We obtain the following theorem. 
\begin{theorem}\label{th:1}
Let $P \in \mathcal{P}_{\sm}$ of dimension $n$. There is an isomorphism of $\Q$-graded algebras
\[
\varphi_P \colon \Pi(P) \simeq \overline{R}_{\Sigma}
\]
given by 
\[
[Q] \longmapsto \exp(h_Q).
\]
\end{theorem}
\begin{proof}
This follows from a combination of Theorem \ref{th:iso} and Theorem \ref{th:chow}. Indeed, take $\varphi_P = \alpha_{\Sigma_P}\circ \Theta_P$.
\end{proof}

We now see that the combinatorial push-forward map from Theorem \ref{th:comb-push} is compatible with the push-forward map between the Chow groups. The following theorem is proven in \cite[Theorem 2.3]{Brion1}.
\begin{theorem}\label{th:comm1}
Let $\Sigma' \geq \Sigma \in W_{\sm}$ be non-degenarate, smooth fans. Then the following diagram commutes.
\begin{figure}[H]
\begin{center}
    \begin{tikzpicture}
      \matrix[dmatrix] (m)
      {
        \overline{R}_{\Sigma'} & A^*\left(X_{\Sigma'}\right)_{\Q}\\
       \overline{R}_{\Sigma} & A^*\left(X_{\Sigma}\right)_{\Q} \\
      };
      \draw[->] (m-1-1) to node[above]{$\alpha_{\Sigma'}^{-1}$} (m-1-2);
      \draw[->] (m-2-1) to node[below]{$\alpha_{\Sigma}^{-1}$} (m-2-2);
      \draw[->] (m-1-1) to node[left]{$\pi_{\Sigma', \Sigma}$} (m-2-1);
      \draw[->] (m-1-2) to node[right]{$\pi_*$} (m-2-2);
     \end{tikzpicture}
     \end{center}
   
     \end{figure}
Here, the maps $\alpha_{\Sigma'}, \alpha_{\Sigma}$ denote the isomorpisms from Theorem \ref{th:chow} and $\pi_{\Sigma', \Sigma}$ the map given in Theorem \ref{th:comb-push} and $\pi_*$ the push-forward map between the Chow groups. 
     \end{theorem}
 Combining this with Theorem \ref{th:1}, we obtain the following corollary.    
     \begin{cor}\label{cor:commut}
Let $P_2 \geq P_1$ be two full-dimensional polytopes in $\mathcal{P}_{\sm}$. Then there exists a map $g \colon \Pi(P_2) \to \Pi(P_1)$ making the following diagram commute. 

\begin{figure}[h]
\begin{center}
    \begin{tikzpicture}
      \matrix[dmatrix] (m)
      {
        \Pi(P_2) & A^*\left(X_{\Sigma_{P_2}}\right)_{\Q}\\
       \Pi(P_1) & A^*\left(X_{\Sigma_{P_1}}\right)_{\Q} \\
      };
      \draw[->] (m-1-1) to node[above]{$\Theta_2$} (m-1-2);
      \draw[->] (m-2-1) to node[below]{$\Theta_1$} (m-2-2);
      \draw[->] (m-1-1) to node[left]{$g$} (m-2-1);
      \draw[->] (m-1-2) to node[right]{$\pi_*$} (m-2-2);
     \end{tikzpicture}
     \end{center}
     \end{figure}
Here, the maps $\Theta_i= \Theta_{P_i}$ denote the isomorpisms from Theorem \ref{th:iso}. Hence, we obtain an isomorphsim of $\Q$-vector spaces 
\[
\Theta \colon \varprojlim_{P \in \mathcal{P}_{\sm}} \Pi(P) \longrightarrow A^*\left(\mathfrak{X}^{\tor}\right)_{\Q},
\]
where the maps on the left hand side are given by the combinatorial push-forward morphisms $g$.
\end{cor}
\begin{proof}
Consider the map $g \colon \Pi(P_2) \to \Pi(P_1)$ defined on the distinguished set of generators \eqref{eq:gen} by
\[
[Q] \longmapsto \varphi_1^{-1} \circ \pi_{\Sigma_2, \Sigma_1} \circ \varphi_2 [Q],
\]
 where $\varphi_i = \varphi_{P_i}$ denotes the isomorphism from Theorem \ref{th:1}.
Then $g$ satisfies the desired property. 
\end{proof}

\section{The embedding}\label{sec:emb}
Let notations be as in Sections \ref{sec:not} and \ref{sec:conv-set}. The goal of this section is to relate the convex-set algebra $\mathcal{C}$ with the toric b-Chow group of the toric riemann--Zariski Space.

We start by defining a map 
\begin{align}\label{eq:uni}
\gamma \colon \mathcal{C} \longrightarrow \varprojlim_{P \in \mathcal{P}_{\sm}} \Pi(P)
\end{align}
of $\Q$-vector spaces.
We proceed in three steps. 

\textbf{Step 1} Fix $P \in \mathcal{P}_{\sm}$ of dimension $n$.  We define a morpism of $\Q$-vector spaces 
\[
\gamma_P \colon \mathcal{C} \longrightarrow \Pi(P)
\]
as follows. For a compact convex set $K \in \mathcal{K}$, consider the corresponding nef toric $b$-divisor $\D_K = \left(D_{K,\Sigma'}\right)_{\Sigma' \geq \Sigma_P}$ on $X_{\Sigma_P}$ (Remark \ref{rem:convex-b}). Now, let $P_1 \geq P$ in $\mathcal{P}_{\sm}$ be sufficiently large such that $D_{K,\Sigma_{P_1}}$ is nef on $X_{\Sigma_{P_1}}$. Then set 
\[
\gamma_P \left([K]\right) = g_1\left(\left[P_{D_{K,\Sigma_{P_1}}}\right]\right),
\]
where $g_1 \colon \Pi(P_1) \to \Pi(P)$ denotes the combinatorial push-forward from Corollary~\ref{cor:commut}. Note that the fact that $D_{K,\Sigma_{P_1}}$ is nef on $X_{\Sigma_{P_1}}$ implies that the class of its corresponding polytope $P_{D_{K,\Sigma_{P_1}}}$ is indeed an element of $\Pi(P_1)$. 

To see that this does not depend on the choice of $P_1$, let $P_2 \geq P$ be another polytope in $\mathcal{P}_{\sm}$ such that $D_{K, \Sigma_{P_2}}$ is also nef in $X_{\Sigma_{P_2}}$. Consider a common refinement $P_3 \geq P_1$, $P_3 \geq P_2$ in $\mathcal{P}_{\sm}$. Then, it follows from item (4) in Theorem \ref{th:comb-push} that the following diagram is commutative. 

\begin{center}
    \begin{tikzpicture}
      \matrix[dmatrix] (m)
      {
        & \Pi(P_3) & \\
       \Pi(P_1) & & \Pi(P_2) \\
       & \Pi(P) & \\
      };
      \draw[->] (m-1-2) to node[left]{$g_4$} (m-2-1);
      \draw[->] (m-1-2) to node[right]{$g_5$} (m-2-3);
      \draw[->] (m-2-1) to node[left]{$g_1$} (m-3-2);
      \draw[->] (m-2-3) to node[right]{$g_2$} (m-3-2);
      \draw[->] (m-1-2) to node[right]{$g_3$} (m-3-2);

     \end{tikzpicture}
   \end{center}
By our choice of $P_1$ and $P_2$ and the fact that $\D_K$ is a b-divisor, we have that the push-forward of $D_{K, \Sigma_{P_3}}$ along both toric birational morphisms $X_{\Sigma_{P_3}} \to  X_{\Sigma_{P_1}}$ and $X_{\Sigma_{P_3}} \to  X_{\Sigma_{P_2}}$ are $D_{K, \Sigma_{P_1}}$ and $D_{K, \Sigma_{P_2}}$, respectively, which are again nef. Hence, on the combinatorial side, we get 
\[
g_4\left(\left[P_{D_{K, \Sigma_{P_3}}}\right]\right) = \left[P_{D_{K, \Sigma_{P_1}}}\right] \text{ and } g_5\left(\left[P_{D_{K, \Sigma_{P_3}}}\right]\right) = \left[P_{D_{K, \Sigma_{P_2}}}\right].
\]
 Thus, we obtain
\[
 g_1\left(\left[P_{D_{K, \Sigma_{P_1}}}\right]\right) 
 = g_1 \circ g_4\left(\left[P_{D_{K, \Sigma_{P_3}}}\right]\right)
 = g_3\left(\left[P_{D_{K, \Sigma_{P_3}}}\right]\right) 
 = g_2 \circ g_5\left(\left[P_{D_{K, \Sigma_{P_3}}}\right]\right)
 = g_2\left(\left[P_{D_{K, \Sigma_{P_2}}}\right]\right),
 \]
 as we wanted to show.
  
  Finally, for two compact convex sets $K_1, K_2 \in \mathcal{K}$, and for any rational number $t \in \Q$, we set
  \[
  \gamma_P\left([K_1] + t[K_2]\right) = \gamma_P ([K_1]) + t \gamma_P ([K_2]).
  \]
Moreover, note that $\gamma_P$ satisfies
\[
\gamma_P(\left[K_1\right]) + \gamma_P(\left[K_2\right]) = \gamma_P(\left[K_1 \cup K_2\right]) + \gamma_P(\left[K_1 \cap K_2 \right]),
\]
whenever $K_1 \cup K_2 \in \mathcal{K}$, and also 
\[
\gamma_P(\left[K+t\right]) + \gamma_P([K]).
\]
Indeed, this follows since it is true for the combinatorial push-forward.

Thus, we obtain a homomorphsim
\[
\gamma_P \colon \mathcal{C} \longrightarrow \Pi(P)
\]
of $\Q$-vector spaces. 

\textbf{Step 2} Let $P'' \geq P'$ in $\mathcal{P}_{\sm}$ of dimension $n$. We will see that the following diagram commutes. 
\begin{figure}[H]
\begin{center}
    \begin{tikzpicture}
      \matrix[dmatrix] (m)
      {
        & \mathcal{C} & \\
       \Pi(P'') & & \Pi(P') \\
      };
      \draw[->] (m-1-2) to node[left]{$\gamma_{P''}$} (m-2-1);
      \draw[->] (m-1-2) to node[right]{$\gamma_{P'}$} (m-2-3);
      \draw[->] (m-2-1) to node[above]{$g$} (m-2-3);
    
     \end{tikzpicture}
   \end{center}
  \end{figure}
It suffices to show this for generators $[K]$, $K \in \mathcal{K}$. Let $[K] \in \mathcal{C}$ such a generator and let $\D_K$ be the corresponding nef toric $b$-divisor. Choose a polytope $P''' \geq P'' \geq P'$ in $\mathcal{P}_{\sm}$ such that $D_{K, \Sigma_{P'''}}$ is nef on $X_{\Sigma_{P'''}}$. Consider the following diagram. 
 \begin{center}
    \begin{tikzpicture}
      \matrix[dmatrix] (m)
      {
        & \mathcal{C} & \\
       \Pi(P'') & & \Pi(P') \\
       & \Pi(P''') & \\
      };
      \draw[->] (m-1-2) to node[left]{$\gamma_{P''}$} (m-2-1);
      \draw[->] (m-1-2) to node[right]{$\gamma_{P'}$} (m-2-3);
      \draw[->] (m-2-1) to node[above]{$g$} (m-2-3);
      \draw[->] (m-3-2) to node[left]{$g''$} (m-2-1);
      \draw[->] (m-3-2) to node[right]{$g'$} (m-2-3);
    
     \end{tikzpicture}
   \end{center} 
We already know that the triangle below commutes. Hence, we get
\[
g \circ \gamma_{P''}([K]) = g \circ g''\left(\left[P_{D_{K,\Sigma_{P'''}}}\right]\right) \\
= g'\left(\left[P_{D_{K,\Sigma_{P'''}}}\right]\right) = \gamma_{P'}([K]),
\]
as we wanted to show. 
  
\textbf{ Step 3} By the universal property of the inverse limit, there exists a (unique) homomorphism of $\Q$-vector spaces 
\[
\gamma \colon \mathcal{C} \to \varprojlim_{P \in \mathcal{P}_{\sm}} \Pi(P),
\]
making the following diagram commute whenever $P_2 \geq P_1 \in \mathcal{P}_{\sm}$ are of dimension $n$.
\begin{center}
    \begin{tikzpicture}
      \matrix[dmatrix] (m)
      {
        & \mathcal{C} & \\
        & \varprojlim_{P \in \mathcal{P}_{\sm}} \Pi(P) & \\
       \Pi(P_2) & & \Pi(P_1) \\
      };
      \draw[->,bend right] (m-1-2) to node[left]{$\gamma_{P_2}$} (m-3-1);
      \draw[->, bend left] (m-1-2) to node[right]{$\gamma_{P_1}$} (m-3-3);
      \draw[->] (m-3-1) to node[above]{$g$} (m-3-3);
      \draw[->] (m-1-2) to node[right]{$\gamma$} (m-2-2);
      \draw[->] (m-2-2) to node[left]{$\pi_{P_2}$} (m-3-1);
       \draw[->] (m-2-2) to node[right]{$\pi_{P_1}$} (m-3-3);
    
     \end{tikzpicture}
   \end{center} 
Here, the maps $\pi_{P_i}$ denote the canonical projection.

\begin{theorem}\label{th:injection}
The map $\gamma$ constructed above is injective. Hence, the map 
\begin{align}
\iota \colon \mathcal{C} \hookrightarrow A^*\left(\mathfrak{X}^{\tor}\right)_{\Q},
\end{align}
defined by $\iota \coloneqq \Theta \circ \gamma$ is an injective homomorphism of $\Q$-vector spaces such that
\[
[K] \longmapsto \exp(\D_K)
\]
for $K \in \mathcal{K}$. Here, $\D_K$ is the nef toric b-divisor corresponding to the convex set $K$ and $\Theta$ is the isomorphism from Corollary~\ref{cor:commut}
\end{theorem}
\begin{proof}
Let notations be as in the construction of the map $\gamma$. It suffices to show injectivity for generators of the form $[K]$, $K \in \mathcal{K}$. Let $[K_1], [K_2]$ be two such generators and suppose that $\gamma([K_1]) = \gamma([K_2])$. Then, for all $P \in \mathcal{P}_{\sm}$, we have that
\[
\gamma_{P}([K_1]) = \pi_{P}([K_1]) = \pi_P([K_2]) = \gamma_P([K_2]).
\]
It follows that 
\[
\left(D_{\gamma_{P}[K_1]}\right)_{P \in \mathcal{P}_{\sm}} \quad \text{and} \quad \left(D_{\gamma_{P}[K_2]}\right)_{P \in \mathcal{P}_{\sm}}
\]
define the same nef toric b-divisor. Thus, since there is a bijective correspondence between convex sets in $\mathcal{K}$, modulo translations, and nef toric b-divisors (Remark \ref{rem:convex-b}), we conclude that $[K_1] = [K_2]$. 
\end{proof}
We endow $\iota(\mathcal{C}) \subset A^*\left(\mathfrak{X}^{\tor}\right)_{\Q}$ with the $\Q$-algebra structure induced by $\mathcal{C}$. Then we obtain the following corollary. 
\begin{cor}\label{cor:inj}
\begin{enumerate}
\item The map $\iota \colon \mathcal{C} \to \iota(\mathcal{C})$ is an isomorphism of $\Q$-graded algebras. Under this isomorphism, the class of a nef toric b-divisor $\D$ is identified with $\log\left([K_{\D}]\right)$, where $K_{\D} \in \mathcal{K}$ denotes its corresponding compact convex set. 
\item The restriction to the polytope algebra coincides with the isomorphism from \ref{eq:poly-iso}.
\end{enumerate}
\end{cor}
\begin{proof}
This follows from the definition of the map $\iota$ and the algebra structure defined on $\iota(\mathcal{C})$.
\end{proof}
We also obtain the following structure result for $\mathcal{C}$.

\begin{cor}\label{cor:inj2} 
The convex-set algebra $\mathcal{C}$ has the structure of an infinite dimensional commutative graded $\Q$-algebra
\[
\mathcal{C} = \mathcal{C}_0 \oplus \mathcal{C}_1 \oplus \dotsc \oplus \mathcal{C}_n.
\] 
Moreover, we have 
\begin{enumerate}
\item $\mathcal{C}_0 \simeq \Q$ and $\mathcal{C}_n \subset \R$.
\item $\mathcal{C}_i \mathcal{C}_j \subset \mathcal{C}_{i +j}$ for all $1 \leq i, j \leq n$ and $i + j\leq n$.
\item $\mathcal{C}$ is generated by $\mathcal{C}_1$ as a $\Q$-algebra.
\end{enumerate}
\end{cor}
\begin{proof}
Recall that $\mathcal{C} = \oplus_{\ell = 0}^{\infty}\mathcal{C}_{\ell}$ is a graded $\Q$-algebra, where the $\ell$'th graded piece is given by all elements of the form $\left(\log[K]\right)^{\ell}$ for $K \in \mathcal{K}$. 

Now, recall that the b-Chow group $A^*\left(\mathfrak{X}^{\tor}\right)_{\Q}$ decomposes as a direct sum 
\[
A^*\left(\mathfrak{X}^{\tor}\right)_{\Q} = \bigoplus_{\ell}^nA^{\ell}\left(\mathfrak{X}^{\tor}\right)_{\Q},
\]
where $A^{\ell}\left(\mathfrak{X}^{\tor}\right)_{\Q} = \varprojlim_{\Sigma \in W'_{\sm}}A^{\ell}\left(X_{\Sigma}\right)$. Then, by definition, the map $\iota$ preserves the grading in the sense that 
\[
\iota\left(\mathcal{C}_{\ell}\right) \subset A^{\ell}\left(\mathfrak{X}^{\tor}\right)_{\Q}
\]
for all $0 \leq \ell \leq n$.
It follows that the graded components $\mathcal{C}_{\ell}$ vanish in all degrees $\ell > n$. Moreover, we have $\mathcal{C}_0 \simeq \Q$ (by definition of $\mathcal{C}$) and $\mathcal{C}_n \subset \varprojlim_{\Sigma}\Q = \R$. Item $(2)$ is clear and $(3)$ follows since $[K] = \exp(\log[K])$ for any $K \in \mathcal{K}$. 
\end{proof}
\begin{rem} By Corollaries \ref{cor:inj} and \ref{cor:inj2},  the convex-set algebra can be viewed as a ring for an intersection theory of sufficiently positive toric b-classes on the toric b-Chow group. 
\end{rem}


\begin{Def}
Let $x \in \mathcal{C}_n$. The \emph{degree} of $x$ is the real number $\iota(x) \in \R$. We denote it by $\deg(x)$. 
\end{Def}
The following lemma states that in the case of nef toric b-divisors, the top intersection product on $\iota(\mathcal{C})$ induced by the multiplication on $\mathcal{C}$ coincides with the top intersection product defined in \cite{botero} (see \ref{sec:b-divisors}).
\begin{lemma}\label{lem:volume}
Let $\D_1, \dots, \D_n$ be nef toric b-divisors on $X$ and let $K_1, \dotsc, K_n$ be the corresponding compact convex sets. For $i = 1, \dotsc, n$, set $k_i =  \log\left([K_i]\right)$. Then 
\[
\deg(k_1\dotsm k_n) = \D_1 \dotsm \D_n = \MV \left(K_1, \dotsc, K_n \right).
\]
\end{lemma}
\begin{proof}
We have $\iota(k_i) = \D_i$ for $i = 1, \dotsc, n$. Then
\[
\deg(k_1\dotsm k_n) = \iota(k_1 \dotsm k_n) = \iota(k_1) \dotsm \iota(k_n) = \D_1 \dotsm \D_n = \MV \left(K_1, \dotsc, K_n \right),
\]
where the last equality is Theorem \ref{the:volume}.
\end{proof} 
\begin{rem}
The previous lemma generalizes the results in \cite[Section 5.3]{Brion}, which show that the top intersection number of top-degree elements in the polytope algebra are given as mixed volumes of polyhedra.
\end{rem}

\section{Hodge type inequalities}\label{sec:applications}
Let notations be as in Sections \ref{sec:not} and \ref{sec:conv-set}. As an application of the results in the previous section, we show that some Hodge type inequalities are satisfied for the convex set algebra $\mathcal{C}$.


We start with the following Alexandrov--Fenchel type inequality. 

\begin{theorem}\label{th:alexandrov}
Let $K_1, \dotsc, K_n$ be compact convex sets in $\mathcal{K}$ and set $k_i = \log([K_i]) \in \mathcal{C}$ for $i = 1, \dotsc, n$. Then the following inequality holds true. 
\begin{align}\label{eq:alex-fen}
\deg(k_1 \dotsm k_n)^2 \geq \deg(k_1\cdot k_1 \cdot k_3 \dotsm k_n)\deg(k_2\cdot k_2 \cdot k_3 \dotsm k_n). 
\end{align}
In particular, for $n=2$, we get 
\[
\deg\left(k_1\cdot k_2\right)^2 \geq \deg\left(k_1^2\right) \deg\left(k_2^2\right) .
\]

\end{theorem}
\begin{proof}
This follows from Lemma \ref{lem:volume} together with the Alexandrov--Fenchel inequality for convex bodies \cite[Theorem 7.3.1]{BM}
\end{proof}
The following is a Hodge index theorem for $\mathcal{C}$ in the case $n=2$. 
\begin{cor}
Let $n = 2$ and let $K$ be a compact convex body in $\mathcal{K}$. Set $k = \log([K]) \in \mathcal{C}$ and suppose that $\deg(k^2) > 0$. Let $L$ be any other compact convex set in $\mathcal{K}$ such that $\deg(\ell \cdot k) = 0$, where $\ell = \log([L]) \in \mathcal{C}$. Then we have that $\deg(\ell^2) = 0$.

\end{cor}
\begin{proof}
We know that $\deg(\ell ^2) \geq 0$. Sppose that $\deg(\ell^2) > 0$. 
Then, by Theorem~\ref{th:alexandrov} we have that 
\[
0 = \deg(\ell \cdot k)^2 \geq \deg(\ell^2)\deg(k^2) > 0,
\]
a contradiction.  
\end{proof}
The next is a generalized inequality of Hodge type for $\mathcal{C}$.
\begin{theorem}\label{th:gen-hodge}
Let $1 \leq p \leq n$ and let $K_1, \dotsc, K_p, L_1, \dotsc, L_{n-p}$ be compact convex sets in $\mathcal{K}$. Set $k_i = \log([K_i]) \in \mathcal{C}$ for $i = 1, \dotsc, p$ and $\ell_i = \log([L_i]) \in \mathcal{C}$ for $i = 1, \dotsc, n-p$. Then, the following inquality holds true. 
\begin{align}\label{eqn:ineq0}
\deg\left(k_1 \dotsm k_p \cdot \ell_1 \dotsm \ell_{n-p}\right)^p \geq \deg\left(\left(k_1\right)^p \cdot \ell_1\dotsm \ell_{n-p}\right) \dotsm \deg\left(\left(k_p\right)^p\cdot \ell_1\dotsm \ell_{n-p}\right).
\end{align}
In particular, for $p = n$, we get 
\[
\deg\left(k_1 \dotsm k_n\right)^n \geq \deg\left(k_1^n\right)  \dotsm \deg\left(k_n^n\right).
\]
\end{theorem}
\begin{proof}
We proceed the proof by induction on $p$. 

For $p=1$, the claim is clear. For $p=2$, the result follows from Theorem~\ref{th:alexandrov}. Now, assume that $p \geq 3$ and suppose the claim is true for $p-1$. We will show the result for $p$. 

Let $A_1, \dotsc, A_{p-1}, H$ be arbitrary compact convex sets in $\mathcal{K}$. Set $a_i = \log([A_i])$ for $i = 1, \dotsc p-1$ and $h = \log([H])$. We claim that
\begin{align}\label{eqn:ineq2}
\deg\left(a_1 \dotsm a_{p-1} \cdot h  \cdot \ell_1 \dotsm \ell_{n-p}\right)^{p-1} \geq \prod_{i =1}^{p-1}\deg\left(a_i^{p-1}\cdot h \cdot \ell_1 \dotsm \ell_{n-p}\right).
\end{align}
In order to see this, fix a fan $\Sigma \in W'_{\sm}$. We may consider the nef toric b-divisor $\pmb{D}_H$ corresponding to $H$, and its incarnation $D_{\Sigma}=D_{H,\Sigma}$ on $X_{\Sigma}$. This is a (not necessarily nef) toric divisor on $X_{\Sigma}$. Note that $D_{\Sigma}$ is a toric subvarity of codimension one. Using the moving lemma on $X_{\Sigma}$, we may suppose that the toric b-divisors $\D_{A_i}$ for $i = 1, \dotsc, p-1$ and $\D_{L_i}$ for $i = 1, \dotsc, n-p$, corresponding to the convex sets $A_i$ and $L_i$, respectively, restrict to $D_{\Sigma}$. We write $\D_{A_i}|_{\Sigma}$ and $\D_{L_i}|_{\Sigma}$ for this restriction. Then, using Lemma \ref{lem:volume} and the induction hypothesis, we get that the inequality 
\begin{align}\label{eq1}
\left(\D_{A_1} \dotsm \D_{A_{p-1}} \cdot D_{\Sigma} \cdot \D_{L_1} \dotsm \D_{L_{n-p}}\right)^{p-1} \geq \prod_{i =1}^{p-1}\left(\left( \D_{A_i} \right)^{p-1} \cdot D_{\Sigma} \cdot \D_{L_1} \dotsm \D_{L_{n-p}}\right), 
\end{align}
which is equivalent to 
\[
\left(\D_{A_1}|_{\Sigma} \dotsm \D_{A_{p-1}}|_{\Sigma} \cdot \D_{L_1}|_{\Sigma} \dotsm \D_{L_{n-p}}|_{\Sigma}\right)^{p-1} \geq \prod_{i =1}^{p-1}\left(\left( \D_{A_i}|_{\Sigma} \right)^{p-1} \cdot \D_{L_1}|_{\Sigma} \dotsm \D_{L_{n-p}}|_{\Sigma}\right),
\]
is true. The same argument works for any $\Sigma \in W'_{\sm}$. Hence, taking limits in \eqref{eq1}, we obtain 
\begin{align*}
\left(\D_{A_1} \dotsm \D_{A_{p-1}} \cdot \D_{H} \cdot \D_{L_1} \dotsm \D_{L_{n-p}}\right)^{p-1} \geq \prod_{i =1}^{p-1}\left(\left( \D_{A_i} \right)^{p-1} \cdot \D_H \cdot \D_{L_1} \dotsm \D_{L_{n-p}}\right),
\end{align*}
which, again using Lemma \ref{lem:volume}, is equivalent to \eqref{eqn:ineq2}.

Now we see that \eqref{eqn:ineq0} follows from \eqref{eqn:ineq2}. Fix some index $s \in \{1, \dotsc, p\}$ and apply \eqref{eqn:ineq2} with $H = K_s$ and $A_1, \dotsc, A_{p-1}$ the remaining $K_i$'s. We get 
\[
\deg\left(k_1 \dotsm k_p \cdot \ell_1 \dotsm \ell_{n-p}\right)^{p-1} \geq \prod_{i \neq s} \deg\left(k_i^{p-1} \cdot k_s \cdot \ell_1 \dotsm \ell_{n-p}\right).
\]
Taking the product over $s$ yields 
\begin{align}\label{eqn:ineq3}
\deg\left(k_1 \dotsm k_p \cdot \ell_1 \dotsm \ell_{n-p}\right)^{p(p-1)} \geq \prod_s\prod_{i \neq s} \deg\left(k_i^{p-1} \cdot k_s \cdot \ell_1 \dotsm \ell_{n-p}\right).
\end{align}
On the other hand, applying \eqref{eqn:ineq2} with $H=A_1= \dotsc = A_{p-2} = K_i$ and $A_{p-1} = K_s$, we obtain 
\[
\deg\left(k_i^{p-1} \cdot k_s \cdot \ell_1 \dotsm \ell_{n-p}\right)^{p-1} \geq 
\deg\left(k_i^p \cdot \ell_1 \dotsm \ell_{n-p}\right)^{p-2}\deg\left(k_s^{p-1} \cdot k_i \cdot \ell_1 \dotsm \ell_{n-p}\right).
\]
Therefore, 
\begin{align*}
& \prod_s \prod_{i \neq s} \deg\left(k_i^{p-1} \cdot k_s \cdot \ell_1 \dotsm \ell_{n-p}\right)^{p-1} \\
& \geq \prod_s \prod_{i \neq s} \deg\left(k_i^p \cdot \ell_1 \dotsm \ell_{n-p}\right)^{p-2}\deg\left(k_s^{p-1}\cdot k_i \cdot \ell_1 \dotsm \ell_{n-p}\right) \\
&= \left(\prod_i\deg\left(k_i^p \cdot \ell_1 \dotsm \ell_{n-p}\right)^{(p-1)(p-2)}\right)\left(\prod_s\prod_{i \neq s} \deg\left(k_i^{p-1}\cdot k_s \cdot \ell_1 \dotsm \ell_{n-p}\right)\right).
\end{align*}
The second term on the right cancels against the left-hand side, and taking $(p-2)$'th roots we obtain
\begin{align}\label{eqn:ineq4}
\prod_s\prod_{i \neq s} \deg\left(k_i^{p-1} \cdot k_s \cdot \ell_1 \dotsm \ell_{n-p}\right) \geq \prod_s\deg\left(k_s^p \cdot \ell_1 \dotsm \ell_{n-p}\right)^{p-1}.
\end{align}
Plugging \eqref{eqn:ineq4} into \eqref{eqn:ineq3} we get
\[
\deg\left(k_1 \dotsm k_p \cdot \ell_1 \dotsm \ell_{n-p}\right)^{p(p-1)} \geq \prod_s\deg\left(k_s^p \cdot \ell_1 \dots \ell_{n-p}\right)^{p-1}.
\]
Finally, taking $(p-1)$'th roots we obtain the result. 

\end{proof}

We have the following corollary.
\begin{cor} 
Let $K$ and $L$ be two compact convex sets in $\mathcal{K}$. Set $k = \log([K]) \in \mathcal{C}$ and $\ell = \log([L]) \in \mathcal{C}$. Then the following inequalities are satisfied. 
\begin{enumerate}
\item\label{it1} For any integers $1\leq q \leq p \leq n$ 
\[
\deg\left(k^q \cdot \ell^{n-q}\right)^p \geq \deg\left(k^p \cdot \ell^{n-p} \right)^q \deg\left(\ell^n\right)^{p-q}
\]
\item\label{it2} For any $1 \leq i \leq n$ 
\[
\deg\left(k^i\cdot \ell^{n-i}\right)^n \geq \deg\left(k^n\right)^i \deg\left(\ell^n\right)^{n-i}.
\]
\item\label{it3} 
\[
\deg\left(\left(k+\ell\right)^n\right)^{\frac{1}{n}} \geq \deg\left(k^n\right)^{\frac{1}{n}} + \deg\left(\ell^n\right)^{\frac{1}{n}}.
\]
\end{enumerate}
\end{cor}
\begin{proof}
For \eqref{it1}, take $K_1 = \dotsc = A_q = K$ and $K_{q+1} = \dotsc = K_p = L_1 = \dotsc = L_{n-p} = L$ in Theorem \ref{th:gen-hodge}.

\eqref{it2} is just \eqref{it1} with $q=i$ and $p =n$. 

For \eqref{it3}, expand $\left(k + \ell \right)^n$, apply \eqref{it2} and take $n$'th rooths. 
\end{proof}
\begin{exa}
Let $K_1, K_2$ be compact convex sets in $\mathcal{K}$ and set $k_i = \log([K_i]) \in \mathcal{C}$ for $i = 1,2$. Define the sequence of numbers
\[
b_j \coloneqq \log\left(\deg\left(k_1^j k_2^{n-j}\right)\right).
\]
for $ j = 0, \dotsc, n$. 
Then
\[
b_{j-1} + b_{j+1} \leq 2b_j,
\]
and thus, the sequence $\left(b_j\right)_{j = 0} ^n$ is concave.
\end{exa}

\printbibliography

\noindent Ana María Botero\\
Institut für Mathematik\\
Universität Regensburg\\
Universitätsstr. 31\\
93053 Regensburg\\
Germany\\
e-mail: \url{ana.botero@mathematik.uni-regensburg.de}
\end{document}